\documentclass[leqno]{macrorend}

\volumeyear{xxxx}\yearnumber{x}\volumenumber{xx}

\begin{document}

\selectlanguage{english}

\newcommand{\R}{\mathbb{R}}
\newcommand{\Z}{\mathbb{Z}}
\newcommand{\N}{\mathbb{N}}

\renewcommand{\epsilon}{\varepsilon}
\newcommand{\eps}{\varepsilon}
\newcommand{\F}{\mathcal{F}}
\newcommand{\E}{\mathcal{E}}

\renewcommand{\leq}{\leqslant}
\renewcommand{\le}{\leqslant}
\renewcommand{\geq}{\geqslant}
\renewcommand{\ge}{\geqslant}

\renewcommand{\emptyset}{\varnothing}

\articolo[Nonlocal phase transitions]{Nonlocal phase transitions: \\ rigidity results
and anisotropic geometry}{Serena Dipierro\footnotemark[1],
Joaquim Serra\footnotemark[2] and
Enrico Valdinoci\footnotemark[3]}

\footnotetext[1]{School of Mathematics and Statistics,
University of Melbourne,
Richard Berry Building,
Parkville VIC 3010,
Australia.
{\tt sdipierro@unimelb.edu.au} }

\footnotetext[2]{Weierstra{\ss} Institut f\"ur Angewandte Analysis
und Stochastik, Mohrenstrasse 39, 10117 Berlin, Germany, and
Eidgen\"ossische Technische Hochschule Z\"urich,
R\"amistrasse 101,
8092 Zurich,
Switzerland. {\tt joaquim.serra@upc.edu} }

\footnotetext[3]{School of Mathematics and Statistics,
University of Melbourne,
Richard Berry Building,
Parkville VIC 3010,
Australia, and
Weierstra{\ss} Institut f\"ur Angewandte Analysis
und Stochastik, Mohrenstrasse 39, 10117 Berlin, Germany,
and Dipartimento di Matematica, Universit\`a degli studi di Milano,
Via Saldini 50, 20133 Milan, Italy, and
Istituto di Matematica Applicata e Tecnologie Informatiche,
Consiglio Nazionale delle Ricerche,
Via Ferrata 1, 27100 Pavia, Italy
{\tt enrico@mat.uniroma3.it} }

\title{Nonlocal phase transitions: \\ rigidity results
and anisotropic geometry}

\begin{abstract}
We provide a series of rigidity results
for a nonlocal phase transition equation.
The prototype equation that we consider is of the form
$$ (-\Delta)^{s/2} u=u-u^3,$$
with~$s\in(0,1)$.
More generally, we can take into account equations like
$$ L u = f(u),$$
where $f$ is a bistable nonlinearity
and $L$ is an integro-differential operator, possibly
of anisotropic type.

The results that we obtain are an improvement of flatness
theorem and a series of theorems concerning the one-dimensional
symmetry for monotone and minimal solutions,
in the research line dictated by a classical
conjecture of E.~De Giorgi in~\cite{DG}.

Here, we collect a series of pivotal results, of geometric
type, which are exploited in the proofs of the main results
in~\cite{paperone}.
\end{abstract}




\section{Introduction and main results}

In phase coexistence models,
a classical question, which was posed in \cite{DG}, is whether or not
``typical solutions’’ possess one-dimensional symmetry. In the models
driven by semilinear partial differential equations,
this type of problems has a long history, see e.g.
\cite{GG, BCN, AAC, Sav, DKW} and the references therein.
Related problems arise in the theory of quasilinear
equations, see e.g. \cite{CGS, FSV, FV},
and find applications in dynamical systems, see \cite{ERG}.
We refer to \cite{REV} for a review on this topic.\medskip

Recently, similar questions have been posed for a phase transition model
in which the long-range particle interaction is described
by a nonlocal operator of fractional type, 
see \cite{CSM, SV, CS, CC1, CC2}. Similar models
describe also the atom dislocation in some crystals,
see e.g. Section~2 in~\cite{giampi}, and some phenomena
in mathematical biology, see e.g.~\cite{BIO}.
The goal of this paper is to present a series of rigidity and symmetry
results for semilinear problems
driven by nonlocal operators. The results are so general that they can
be applied also in a non-isotropic medium
(but, as far as we know, they are also new in the isotropic case).\medskip

More precisely, we consider a nonlocal Allen-Cahn  equation of the type
\begin{equation*}
L u = f(u)  \quad \mbox{in }\R^n,
\end{equation*}
where~$L$ is an operator of the form
\begin{equation*}
 Lu(x):= \int_{\R^n} \big( u(x)-u(x+y)\big) 
\frac{\mu\big(y/|y|\big)}{|y|^{n+s}}\,dy,\quad\qquad x\in\R^n,
 \end{equation*}
with~$s\in(0,1)$.
The typical example of 
operator comprised by our setting
is the fractional Laplacian (in this case~$L:=(-\Delta)^{s/2}$).
The basic
nonlinearity~$f$ that
we take into account is when $f$ is ``bistable'', i.e.
it is minus the derivative of a
double-well potential (e.g., $f(u)=u-u^3$).
We assume that the measure $\mu$ (which is often called in jargon the
``spectral measure'')
satisfies 
\begin{equation*}
\mu(z)=\mu(-z)\quad\mbox{and}\quad 
\lambda \le \mu(z) \le \Lambda
\quad\mbox{for all }z\in  S^{n-1},
\end{equation*}
for some~$\Lambda\ge\lambda>0$. Given a bounded $\psi\in C^2(\R)$ we define
\begin{equation}\label{1dfraclap}
A \psi(z) := \int_{-\infty}^{+\infty} 
\frac{\psi(z)-\psi( z+\zeta)}{|\zeta|^{1+s}} \,d\zeta, \qquad\quad z\in \R.
\end{equation}
Roughly speaking, the operator~$A$ plays a role
of the one-dimensional fractional Laplacian. In order to
take into account the possible anisotropy of the operator~$L$,
we need to scale~$A$ appropriately in any fixed direction. To this aim,
for $\psi$ as above,  $\omega\in S^{n-1}$, and $h>0$ 
we define, for $x\in\R^n$,  
\begin{equation*}
\bar\psi_{\omega,h}\,(x)  := \psi\left(\omega\cdot \frac{x}{h}\,\right).\end{equation*}
We set $h_L(\omega) :=h$ where $h>0$ satisfies 
\begin{equation*} 
L \bar\psi_{\omega,h} (x)= A \psi\left(\omega\cdot \frac{x}{h}\right) \ 
\mbox{for all }\psi\in C^2(\R)\cap L^\infty(\R).\end{equation*}
We also define
\begin{equation}\label{CAL C} \mathcal C= \mathcal C_L := 
\bigcap_{\omega\in S^{n-1}} \big\{ x\in \R^n\,:\, 
x\cdot \omega \le h_L(\omega) \big\}\end{equation}
and assume that
\begin{equation*}
{\mbox{$\partial\,\mathcal C_L$ is $C^{1,1}$ and strictly convex.}}\end{equation*}
More quantitatively, we assume that  
there exist $\rho'>\rho>0$ such that
\begin{equation}\tag{H1}\label{assumpL}
\mbox{the curvatures of  $\partial\,\mathcal C_L$ are 
bounded above by  $\displaystyle \frac 1\rho$
and below by $\displaystyle  \frac{1}{\rho'}$}.
\end{equation}
Concerning the nonlinearity $f$,
we assume that $f\in C^1\big([-1,1]\big)$ and, for some $\kappa>0$ and $c_\kappa>0$,  
\begin{equation}\tag{H2}\label{assumpf}
f(-1)=f(1)=0 \quad \mbox{and}\quad 
f'(t)<-c_\kappa \quad \mbox{for } 
t\in [-1,-1+\kappa]\cup[1-\kappa, 1].
\end{equation}
Moreover, recalling the setting in \eqref{1dfraclap}, we assume that
\begin{equation}\tag{H3}\label{existslayer}
\mbox{there exists $\phi_0$ satisfying}\quad  
\begin{cases}
A\phi_0= f(\phi_0) \  &\mbox{in }\R,
\\
\phi_0'>0 &\mbox{in }\R,
\\
\phi_0(0) =0,
\\
\displaystyle\lim_{x\to \pm\infty} \phi_0 = \pm 1.
\end{cases}
\end{equation}
The main result
obtained in~\cite{paperone} is the following
improvement of
flatness:

\begin{theorem}\label{thm:improvement}
Assume that $L$ satisfies $\eqref{assumpL}$
and that $f$ satisfies \eqref{assumpf} and \eqref{existslayer}.
Then there exist universal constants~$\alpha_0\in (0,s/2)$, $
p_0\in(2,\infty)$ and~$a_0\in(0,1/4)$ such that the following statement holds.
\smallskip

Let $a\in(0,a_0)$ and $\eps \in (0, a^{p_0})$. 
Let $u: \R^n \rightarrow (-1,1)$ be a solution 
of 
$$ Lu=\varepsilon^{-s}f(u)\quad{\mbox{ in }}\quad B_{
j_a
},$$ with
$$ j_a:=\left\lfloor  \frac{\log
a}{\log(2^{-\alpha_0})}\right\rfloor.$$
Assume
that $0\in \{-1+\kappa \le u \le  1-\kappa\}$
and that
\[ 
\{\omega_j \cdot x\le -a 2^{j(1+\alpha_0)}\}\, \subset\, \{u\le -1+\kappa\} \,\subset \,\{u\le 1-\kappa\}  \,\subset \,\{\omega_j\cdot x\le a 2^{j(1+\alpha_0)}\} \quad \mbox{in }B_{2^j}, 
\]
for any~$j= 
\left\{0,1,2,\dots, j_a \right\}
$
and for some $\omega_j\in S^{n-1}$.

Then, 
\[ 
\left\{\omega\cdot x \le - 
\frac{a}{2^{1+\alpha_0}} \right\} \,\subset \,\{u\le -1+\kappa\} \,\subset\, \{u\le 1-\kappa\}  \,\subset\, \left\{\omega\cdot x\le  \frac{a}{2^{1+\alpha_0}} \right\} \quad \mbox{in }B_{1 /2},
\]
for some $\omega\in S^{n-1}$.
\end{theorem}

Theorem \ref{thm:improvement} says that if
the level sets of the solution are $C^{1,\alpha_0}$-flat from infinity up
to $B_1$, then they are also $C^{1,\alpha_0}$-flat up to $B_{1/2}$,
and so one can dilate the picture once again and repeat the argument
at any small scale towards the origin
(as a matter of fact, suitable scaled iterations of Theorem \ref{thm:improvement}
are given in Corollaries~7.1 and~7.2 of~\cite{paperone}). 
An important consequence of 
Theorem~\ref{thm:improvement}
is related to the one-dimensional symmetry
properties of the solutions. For this,
we say that a function~$u:\R^n\to\R$ is $1$D if
there exist~$\bar u:\R\to\R$
and~$
\bar\omega\in S^{n-1}$ such that~$u(x)=\bar u(\bar\omega\cdot x)$ for any~$x\in\R^n$.

Then, we have the following consequences of Theorem \ref{thm:improvement}:

\begin{theorem}[One-dimensional symmetry for asymptotically flat solutions]\label{C:1}
Assume that $L$ satisfies $\eqref{assumpL}$ and that $f$ 
satisfies \eqref{assumpf} and \eqref{existslayer}.

Let $u$ be a solution of~$
L u = f(u)$ in~$\R^n$. 

Assume that there exists~$
a:(1,\infty) \rightarrow (0,1]$ such that
$a(R)\searrow 0$ as $R\nearrow+\infty$ and 
such that, for all $R>1$, we have that
\[
\{ \omega\cdot x\le -a(R)R\} \subset 
\{u\le -1+\kappa\}
\subset  \{u\le 1-\kappa\} 
\subset \{ \omega\cdot x\le a(R)R\} \quad \mbox{in }\;B_{R},
\]
for some $\omega\in S^{n-1}$,
which may depend on $R$.
Then, $u$ is 1D.
\end{theorem}

We stress that all these results, as far as we know,
are new even for the equation
$(-\Delta)^{s/2} u=u-u^3$,
with~$s\in(0,1)$, which is a particular case of our setting.\medskip

As a matter of fact, we can
consider the concrete case of minimal solutions
of the nonlocal Allen-Cahn equation~$(-\Delta)^{s/2} u=u-u^3$,
with~$s\in(0,1)$. We remark that the energy functional related to such equation
is 
$$ {\mathcal{E}}(u,\Omega):={\mathcal{E}^{\rm Dir}}(u,\Omega)
+\int_\Omega (1-u^2(x))^2\,dx,$$
where
\begin{equation}\label{KIN}
{\mathcal{E}^{\rm Dir}}(u,\Omega):= C_{n,s}\iint_{\R^{2n}\setminus(\R^n\setminus\Omega)^2}
\frac{|u(x)-u(y)|^2}{|x-y|^{n+s}}\,dx\,dy,\end{equation}
for a suitable normalization constant $C_{n,s}>0$. In this setting,
we say that a solution $u$ of~$(-\Delta)^{s/2} u=u-u^3$ is a {\em minimizer} of $\mathcal E$
in~$\R^n$ if
$$ {\mathcal{E}}(u,B)\le {\mathcal{E}}(u+\varphi,B),$$
for any ball~$B\subset\R^n$ and any~$\varphi\in C^\infty_0(B)$.
In this setting, the following results hold true:

\begin{theorem}[One-dimensional symmetry in the plane]
Let~$u$ be a minimizer of $\mathcal E$ 
in~$\R^2$.
Then, $u$ is $1$D.
\end{theorem}

\begin{theorem}[One-dimensional symmetry for monotone solutions
in $\R^3$]
Let~$n\le3$ and~$u$ be a solution of~$(-\Delta)^{s/2} u=u-u^3$
in~$\R^n$.

Suppose that
$$ \frac{\partial u}{\partial x_n}(x)>0
\quad{\mbox{ for any }}x\in\R^n
\qquad{\mbox{ and }}\qquad \lim_{x_n\to\pm \infty} u(x',x_n)=\pm1.$$
Then, $u$ is $1$D.
\end{theorem}

\begin{theorem}[One-dimensional symmetry when $s$ is close to 1]
Let~$n\le 7$. Then, there exists $\eta_n\in(0,1)$ such that
for any~$s\in [1-\eta_n,\,1)$ the following statement holds true.

Let~$u$ be a
minimizer of $\mathcal E$ 
in~$\R^n$. Then, $u$ is $1$D.
\end{theorem}

\begin{theorem}[One-dimensional symmetry for monotone solutions
in $\R^8$ when $s$ is close to~$1$]
Let~$n\le8$.
Then, there exists $\eta_n\in(0,1)$ such that
for any~$s\in [1-\eta_n,\,1)$ the following statement holds true.

Let~$u$ be a solution of~$(-\Delta)^{s/2} u=u-u^3$ in~$\R^n$.

Suppose that             
\begin{equation*}
\frac{\partial u}{
\partial x_n}(x)>0\quad{\mbox{ for any }}x\in\R^n
\qquad{\mbox{ and }}\qquad
\lim_{x_n\to\pm \infty} u(x',x_n)=\pm1.\end{equation*}
Then, $u$ is $1$D.
\end{theorem}

The full proofs of the results mentioned above are given in~\cite{paperone}.
Here, we present the details of a series of pivotal
results of geometric type that will be exploited in~\cite{paperone}.

Related results on symmetry
problems with possible applications
to nonlocal phase transitions
have been recently announced in \cite{OGGI}
and obtained in~\cite{Onew2}
(we stress that the range of fractional parameter
dealt with in~\cite{OGGI, Onew2} is complementary
to the one of this paper and~\cite{paperone}).

\section{Some useful facts on the distance function}\label{APP:AA}

We collect here some ancillary results of elementary
nature from the theory of convex sets and anisotropic distance functions. For
the convenience of the reader,
we give full details of the results
we need, that are
stated in a convenient form for their use
in the forthcoming paper~\cite{paperone}.

For this,
we consider a continuous and degree~$1$
positively homogeneous function ${{h}}:\R^n\to[0,+\infty)$
and the convex set in \eqref{CAL C}.
We assume that the boundary of $\mathcal C$ is of class $C^1$.
The set $\mathcal C $ in our setting plays the role
of an anisotropic ball, and so, for any $r>0$, we set
\begin{equation}\label{crp}
\mathcal C_r(y): = y+r\mathcal C.\end{equation}
This anisotropic ball induces naturally a norm, defined,
for any $p\in\R^n$, by the following formula:
\begin{equation}\label{pKnorma}
\| p\|_{ \mathcal C } := 
\frac{1}{\sup\{ \tau>0 {\mbox{ s.t. }}\tau p\in {\mathcal C}\}}.\end{equation}
We observe that,
in view of~\eqref{CAL C}, \eqref{crp},
and \eqref{pKnorma},
for any $R>0$ and $z_0\in\R^n$, 
\begin{equation}\label{FAC}
{\mathcal{C}}_R(z_0)=\big\{x\in\R^n {\mbox{ s.t. }}
\|x-z_0\|_{{\mathcal{C}}}\le R\big\}.\end{equation}
Also, we have the following elementary inequality of
Cauchy-Schwarz type:

\begin{lemma}\label{CSani}
For any $x$, $y\in\R^n$, we have
that $x\cdot y\le {{h}}(y)\,\|x\|_{\mathcal{C}}$.
\end{lemma}

\begin{proof} If either~$x=0$
or~$y=0$ we are done, so we can suppose that~$y\ne0$ and~$x\ne0$.
We set~$\omega:=\frac{y}{|y|}\in S^{n-1}$
and~$\eta:=\frac{x}{\|x\|_{\mathcal{C}}}$. Then~$\eta\in{\mathcal{C}}$,
thus by~\eqref{CAL C}
$$ \frac{x}{\|x\|_{\mathcal{C}} }\cdot y=
\eta\cdot \omega\,|y|\le {{h}}(\omega)\,|y|
={{h}}\left( \frac{y}{|y|}\right)\,|y|
= {{h}}(y),$$
which gives the desired result.
\end{proof}

Now we show that,
in the terminology of convex geometry, the function~${{h}}$ is
the ``support function'' of the convex body~${\mathcal C}$.

\begin{lemma}\label{LE:DESE}
For any~$\omega\in S^{n-1}$,
\begin{equation}\label{SUPPORT}
{{h}}(\omega)=\sup_{x\in {\mathcal C}} 
x\cdot\omega.\end{equation}
\end{lemma}

\begin{proof} {F}rom~\eqref{CAL C}, we have that, for any~$x\in{\mathcal C}$
and any $\omega\in S^{n-1}$,
$ x\cdot\omega\le {{h}}(\omega)$,
and so
\begin{equation} \label{9yohikjSDTYryteVH}
\sup_{x\in {\mathcal C}} 
x\cdot\omega\le {{h}}(\omega),\end{equation}
for any $\omega\in S^{n-1}$.
In particular, this implies that ${\mathcal C}$ is bounded
in any direction. Therefore,
to check the opposite inequality
to the one in \eqref{9yohikjSDTYryteVH}, and thus
to complete the proof of the desired result,
we can fix $\omega\in S^{n-1}$
and slide a hyperplane with normal
direction $\omega$ till it touches ${\mathcal C}$
at some point $P\in\partial{\mathcal C}$. That is,
we have that for any $x\in{\mathcal C}$
it holds that $\omega\cdot (x-P)\le0$ and so
\begin{equation}\label{9ij79GJ8yij6789xghjk}
\sup_{x\in{\mathcal C}}\omega\cdot x=
\omega\cdot P.\end{equation}
Also, since $P\in\partial{\mathcal C}$, we deduce from \eqref{CAL C} that
there exists $\varpi\in S^{n-1}$ for which 
\begin{equation}\label{9ij79GJ8yij6789xghjk:2}
\varpi\cdot P={{h}}(\varpi).
\end{equation}
Notice that~$\{\varpi\cdot(x-P)=0\}$ is a supporting
hyperplane for~${\mathcal C}$, since, for any~$x\in {\mathcal C}$,
$$ \varpi\cdot (x-P)=\varpi\cdot x-{{h}}(\varpi)\le0,$$
thanks to~\eqref{CAL C}.

Since $\partial {\mathcal{C}}$ has been assumed to be
a~$C^1$ manifold,
the two supporting hyperplanes at~$P$, namely~$\{\omega\cdot (x-P)=0\}$
and~$\{\varpi\cdot(x-P)=0\}$, must coincide, and so~$\omega=\varpi$.

As a consequence of this, recalling~\eqref{9ij79GJ8yij6789xghjk}
and~\eqref{9ij79GJ8yij6789xghjk:2}, we obtain that
$$ \sup_{x\in{\mathcal C}}\omega\cdot x=
\omega\cdot P=\varpi\cdot P={{h}}(\varpi)={{h}}(\omega),$$
as desired.
\end{proof}

As a counterpart of Lemma \ref{CSani}, we also have

\begin{lemma}\label{LEM:JJ}
Let $z_0\in\R^n$, $R>0$ and $z\in\partial {\mathcal{C}}_R(z_0)$.
Let $\omega_0\in S^{n-1}$ be the inner normal 
of $\partial {\mathcal{C}}_R(z_0)$
at the point $z$. Then
$$ \omega_0\cdot (z_0-z)= R\,{{h}}(\omega_0).$$
\end{lemma}

\begin{proof} Let
\begin{equation*}
\zeta:=\frac{z-z_0}{R}\end{equation*} and notice that
we know that
\begin{equation}\label{9uik5fdsfoil}
\zeta\in\partial {\mathcal{C}}.\end{equation}
Also, since~${\mathcal{C}}_R(z_0)$ is convex, we know
that~${\mathcal{C}}_R(z_0)\subset\{x\in\R^n{\mbox{ s.t. }}
\omega_0\cdot(x-z)\ge0\}$ and so
$${\mathcal{C}}
\subset\{y\in\R^n{\mbox{ s.t. }}
\omega_0\cdot(y-\zeta)\ge0\}.$$
Hence, by~\eqref{SUPPORT},
$$ {{h}}(\omega_0)={{h}}(-\omega_0)=
\sup_{y\in {\mathcal C}} (-y\cdot\omega_0)
\le -\zeta\cdot\omega_0.$$
On the other hand, by Lemma~\ref{CSani},
$$ -\zeta\cdot\omega_0\le {{h}}(\omega_0)\,\|-\zeta\|_{\mathcal{C}}
={{h}}(\omega_0),$$
and so
$$ {{h}}(\omega_0)=-\zeta\cdot\omega_0=
\frac{(z_0-z)\cdot\omega_0}{R},$$
as desired.
\end{proof}
Given a nonempty, closed and convex set~$K\subset\R^n$,
we define the anisotropic signed
distance function from $K$ as
\begin{equation}\label{dk defin}\begin{split}
&d_K(x):=\inf\big\{ \ell(x) \,:\, \ell(x) =
\omega\cdot x + c ,\quad h(\omega)= 1,  \\ &\qquad\qquad\quad c\in\R \quad 
\mbox{and} \quad \ell\ge 0 \  \mbox{ in all of }K\big\}.  
\end{split}\end{equation}
Notice that $d_K$ is a concave function,
since it is the infimum of affine functions.
Also, we have that~$d_K$ is a Lipschitz function,
with Lipschitz constant~$1$ with respect to the anisotropic norm, as stated in the following result:

\begin{lemma}\label{lip lemma app}
For any~$p$, $q\in\R^n$,
$$ |d_K(p)-d_K(q)|\le \| p-q\|_{\mathcal C}.$$
\end{lemma}

\begin{proof} Up to exchanging $p$ and~$q$, we suppose that~$d_K(p)\ge
d_K(q)$. Fixed~$\delta>0$, we let~$\ell_\delta(x)=
\omega_\delta\cdot x + c_\delta$ be such that~$
{{h}}(\omega_\delta)= 1$,  $c_\delta\in\R$,  
$\ell_\delta(x)\ge 0$ for any~$x\in K$, and with~$d_K(q)\ge \ell_\delta(q)-\delta$.
Then, we have that~$d_K(p)\le \ell_\delta(p)$ and so, by Lemma~\ref{CSani},
\begin{eqnarray*}
&& |d_K(p)-d_K(q)|=d_K(p)-d_K(q)\le \ell_\delta(p)-\ell_\delta(q)+\delta\\
&&\qquad=\omega_\delta\cdot (p-q) +\delta\le {{h}}(\omega_\delta)\,
\| p-q\|_{\mathcal C} +\delta=\| p-q\|_{\mathcal C}+\delta.
\end{eqnarray*}
Hence, taking~$\delta$ arbitrarily close to~$0$ we obtain the desired result.
\end{proof}

It is also useful to observe that the infimum
in \eqref{dk defin} is attained, namely:

\begin{lemma}\label{Lem opt dir}
For any $p\in\R^n$ there exists an affine function
$\ell_p$, of the form
$\ell_p(x) =
\omega_p\cdot x + c_p$, with ${{h}}(\omega_p)= 1$,
$c_p\in\R$, such that $\ell_p\ge 0$ in $K$
and $d_K(p)=\ell_p (p)$.

Moreover,
if $t_0\in\R$ and
$z_0\in\{d_K>t_0\}$ are such that $ 
p\in\partial{\mathcal{C}}_R(z_0)\cap
\{d_K=t_0\}$, with
${\mathcal{C}}_R(z_0)\subset\{d_K \ge t_0\}$, and~$\omega_0$
is the interior normal of
${\mathcal{C}}_R(z_0)$ at~$p$, we have that
\begin{eqnarray}
\label{7:AO1}
&&\omega_p=\frac{\omega_0}{{{h}}(\omega_0)}\\
{\mbox{and }}&&c_p=t_0-\frac{\omega_0}{{{h}}(\omega_0)}\cdot p.
\label{7:AO2}\end{eqnarray}
\end{lemma}

\begin{proof} The existence of the optimal
affine function $\ell_p$ follows from the direct methods of the calculus
of variations, so we focus on the proof of the second claim.
We have that, for any $x\in {\mathcal{C}}_R(z_0)$,
$$ \omega_p\cdot p +c_p=
d_K(p)=t_0\le d_K(x)\le \omega_p\cdot x +c_p,$$
that is
$$ \min_{x\in {\mathcal{C}}_R(z_0)} \omega_p\cdot x=
\omega_p\cdot p.$$
Hence, by Lagrange multipliers, the gradient of the map~$\omega_p\cdot x$
is parallel to (and in the same direction of) $\omega_0$, that is
\begin{equation} \label{0ugh8un456}
\omega_p=c\omega_0,\end{equation}
for some~$c\ge0$. Hence,
since~${{h}}$ is homogeneous,
$$ 1={{h}}(\omega_p)=c\,{{h}}(\omega_0).$$
This gives that~$c=\frac{1}{{{h}}(\omega_0)}$, which,
combined with~\eqref{0ugh8un456}, proves \eqref{7:AO1}.
Then, we write $t_0=d_K(p)=\omega_p\cdot p+c_p$ and
we obtain \eqref{7:AO2}.
\end{proof}

In case of tangent anisotropic spheres to level sets
of the anisotropic distance function, a useful comparison
occurs with respect to Euclidean hyperplanes, as stated in the following
result:

\begin{lemma}\label{R5678}
Let $K$ be convex, $z_0\in\{d_K>t_0\}$, $t_0\in\R$. 
Suppose that ${\mathcal{C}}_R(z_0)\subset\{d_K \ge t_0\}$
and let~$z\in \partial {\mathcal{C}}_R(z_0)\cap\{d_K=t_0\}$.

Let~$\omega_0$ be the interior normal of~${\mathcal{C}}_R(z_0)$
at~$z$
and $\{d_K\ge t_0\}\subset \{x\in\R^n{\mbox{ s.t. }}\omega_0
\cdot(x-z)\ge0\}$.
Then, for any~$x\in\R^n$ it holds that
$$ d_K(x)\le \frac{\omega_0}{{{h}}(\omega_0)}\cdot (x-z)+t_0.$$
\end{lemma}

\begin{proof} We let
\begin{equation}\label{d tilde}
\tilde d(x):=\frac{\omega_0}{{{h}}(\omega_0)}\cdot (x-z)+t_0.\end{equation}
We claim that
\begin{equation}\label{POSLA}
{\mbox{$\tilde d(x)\ge0$ for any $x\in K$.}}
\end{equation}
For this, 
we use
Lemma \ref{Lem opt dir}
(with $p=z$), according to which
the affine function
$$ \ell_z(x):=\omega_z\cdot x+c_z,$$
with $
\omega_z=\frac{\omega_0}{{{h}}(\omega_0)}$
and $c_z=t_0-\frac{\omega_0}{{{h}}(\omega_0)}\cdot z$
satisfies $\ell_z(z)=d_K(z)$ and $\ell_z(x)\ge 0$ for any $x\in K$.
In particular, for any $x\in K$,
we have that $\tilde d(x)=\ell_z(x)\ge0$, which
proves~\eqref{POSLA}.

In addition,
from the homogeneity of~${{h}}$ we have that
$$ {{h}}\left( \frac{\omega_0}{{{h}}(\omega_0)}\right)=
\frac{{{h}}(\omega_0)}{{{h}}(\omega_0)}=1.$$
Using this and~\eqref{POSLA}, we obtain the desired result from~\eqref{dk defin}.
\end{proof}

Now we show that
the function~$d_K$, as defined in~\eqref{dk defin}, coincides with the signed distance from
the boundary of $K$, namely:

\begin{proposition}\label{dk defin EQUIVAL}
Let~$K\subset\R^n$ be nonempty, closed and convex.
Then it holds that
\begin{equation}\label{TNAL01}
d_K(x) =
\begin{cases}
+\inf \big\{ \| z-x\|_{\mathcal C}\ :\  z\in\partial K\big\}
\quad &\mbox{for  }  x \in K
\\
-\inf \big\{ \| z-x\|_{\mathcal C}\ :\    z\in\partial K\big\}
&\mbox{for  }  x\in\R^n\setminus K.
\end{cases}
\end{equation}
\end{proposition}

\begin{proof} First, we 
show that
\begin{equation}\label{78:12:01}
{\mbox{$d_K\ge0$ in $K$.}}
\end{equation}
For this, let $\ell$ be any affine function in~\eqref{dk defin}.
Since $\ell\ge0$ in $K$, the claim in \eqref{78:12:01}
plainly follows.

Now we show that
\begin{equation}\label{78:12:02}
{\mbox{$d_K\le0$ in $\R^n\setminus K$.}}
\end{equation}
To this aim, let $p\in
\R^n\setminus K$. Since $K$ is convex, we can separate
it from $p$, namely there exists an affine function
$\ell_o(x)=\omega_o\cdot x+c_o$,
for suitable $\omega_o\in\R^n\setminus\{0\}$ and $c_o\in\R$,
such that $\ell_o\ge0$ in $K$
and $\ell_o(p)\le0$.
So, we define
$$ \omega:=\frac{\omega_o}{{{h}}(\omega_o)},\qquad
c:=\frac{c_o}{{{h}}(\omega_o)}\qquad{\mbox{and}}\qquad
\ell(x):=\omega\cdot x+c.$$
In this way, we have that $\ell(x)=\frac{\ell_o(x)}{
{{h}}(\omega_o)}\ge0$ for any $x\in K$, and $\ell(p)\le0$.
In addition, we have that ${{h}}(\omega)=1$ and so
$\ell$ is an admissible affine function in~\eqref{dk defin}.
This implies that $d_K(p)\le\ell(p)\le0$, which gives
\eqref{78:12:02}.

{F}rom \eqref{78:12:01}, \eqref{78:12:02} and the
continuity of $d_K$ (recall Lemma \ref{lip lemma app}),
it follows that $d_K=0$ along~$\partial K$.
Hence, to complete the proof
of \eqref{TNAL01}, we can restrict to the case in which
$x\not\in\partial K$. Hence, it suffices to check that,
for any $P\not\in\partial K$,
\begin{equation}\label{suh781}
|d_K(P)| =
\inf \big\{ \| z-P\|_{\mathcal C}\ :\  z\in\partial K\big\}.
\end{equation}
To check this, 
we first observe that, from Lemma \ref{lip lemma app}, for any $z\in \partial K$,
$$ |d_K(P)|=|d_K(P)-d_K(z)|\le \| z-P\|_{\mathcal C}$$
and therefore
\begin{equation*}
|d_K(P)| \le
\inf \big\{ \| z-P\|_{\mathcal C}\ :\  z\in\partial K\big\}.
\end{equation*}
Thus, to complete the proof of \eqref{suh781}, we only need
to show that
\begin{equation}\label{suh78100}
|d_K(P)| \ge
\inf \big\{ \| z-P\|_{\mathcal C}\ :\  z\in\partial K\big\}.
\end{equation}
For this, we set $R(P):=
\inf \big\{ \| z-P\|_{\mathcal C}\ :\  z\in\partial K\big\}>0$
and we notice that ${\mathcal C}_{R(P)}(P)$
is contained either in $K$ (if $P\in K$)
or in the closure of the complement of $K$ (if $P\in\R^n\setminus K$),
and there exists $p\in\partial K$ with
$\| p-P\|_{\mathcal C}=R(P)$.

So, if $P\in K$, we use Lemma \ref{Lem opt dir} (with~$z_0=P$
and~$t_0=0$)
to find that the affine 
function $\ell_p(x) := \omega_p\cdot x + c_p$, with 
$\omega_p=\frac{\omega_0}{{{h}}(\omega_0)}$ and 
$c_p=-\frac{\omega_0}{{{h}}(\omega_0)}\cdot p$,
satisfies $\ell_p\ge 0$ in $K$
and $d_K(p)=\ell_p(p)$. 
Moreover, we have that
\begin{equation}\label{0910}
d_K(P)\ge \ell_p(P)
\end{equation}
(and, in fact, equality holds, due to~\eqref{dk defin}).
To check~\eqref{0910}, let~$\omega\in\R^n\setminus\{0\}$,
and set~$\tilde\omega:=\frac{h(\omega)\,\omega}{|\omega|^2}$. We observe that~$\tilde\omega\cdot\omega=
h(\omega)$ and thus~$\tilde\omega\in{\mathcal C}$,
thanks to~\eqref{CAL C}.
In consequence of this and of~\eqref{pKnorma}, we have that
$$ \frac1{\| \tilde\omega \|_{ \mathcal C } }= 
{\sup\{ \tau>0 {\mbox{ s.t. }}\tau \tilde\omega\in {\mathcal C}\}}\ge
1$$
and so
$$ 1\ge 
\|\tilde\omega \|_{ \mathcal C }=
\frac{h(\omega)\, \| \omega \|_{ \mathcal C }}{|\omega|^2}.$$
Therefore, recalling Lemma~\ref{LEM:JJ},
$$ \left\| \frac{h(\omega)}{|\omega|^2}
\frac{\omega_0\cdot(P-p)}{h(\omega_0)}\omega
\right\|_{\mathcal{C}} = 
\frac{| \omega_0\cdot(P-p) |}{ h(\omega_0)}\,
\frac{h(\omega)\,\| \omega\|_{\mathcal{C}}}{|\omega|^2}
= 
R(P)\,\frac{h(\omega)\,\| \omega\|_{\mathcal{C}}}{|\omega|^2} \le R(P),$$
which implies that
\begin{equation}\label{9ws234q8euhf}
P-\frac{h(\omega)}{|\omega|^2}\frac{\omega_0\cdot(P-p)}{h(\omega_0)}\omega\in
{\mathcal{C}}_{R(P)}(P).\end{equation}
Now, let~$\ell(x):=\frac{\omega}{h(\omega)}\cdot x+c$
and suppose that~$\ell\ge0$ in~$K$, and so in particular in~${\mathcal{C}}_{R(P)}(P)$.
Then, from~\eqref{9ws234q8euhf}, we have that
\begin{eqnarray*} 
&& 0\le \ell\left( P-\frac{h(\omega)}{|\omega|^2}\frac{\omega_0\cdot(P-p)}{h(\omega_0)}\omega\right)
=\frac\omega{h(\omega)}\cdot\left( P-\frac{h(\omega)}{|\omega|^2}\frac{\omega_0\cdot(P-p)
}{h(\omega_0)}\omega\right)+c\\
&&\qquad=\frac\omega{h(\omega)}\cdot P
-
\frac{\omega_0\cdot(P-p)
}{h(\omega_0)}+c
=\ell(P)-\ell_p(P).
\end{eqnarray*}
Given the validity of such inequality for every~$\ell$,
we have established~\eqref{0910}.

Accordingly, by~\eqref{0910} and Lemma \ref{LEM:JJ},
$$ |d_K(P)|=d_K(P)\ge \ell_p(P)=
\frac{\omega_0}{{{h}}(\omega_0)}\cdot(P-p)=\|P-p\|_{\mathcal C}=
\inf \big\{ \| z-P\|_{\mathcal C}\ :\  z\in\partial K\big\}.$$
This proves \eqref{suh78100} when $P$ lies inside $K$, so
we now deal with the case in which $P$ lies in $\R^n\setminus K$.

For this, we let $p\in K\cap \partial{\mathcal C}_{R(P)}(P)$
and we denote by~$\omega_0\in S^{n-1}$ the inner normal
of $\partial{\mathcal C}_{R(P)}(P)$ at~$p$. Then, since $K$
is convex, we have that $\omega_0\cdot(x-p)\le0$
for any~$x\in K$. Hence, the affine function
$$ \ell(x):= \frac{-\omega_0}{{{h}}(\omega_0)}\cdot(x-p)$$
satisfies~$\ell\ge0$ in~$K$ and so it is admissible in~\eqref{dk defin}.
Consequently, by Lemma~\ref{LEM:JJ},
$$ -|d_K(P)|=d_K(P)\le \ell(P)=
-\frac{\omega_0}{{{h}}(\omega_0)}\cdot(P-p)=-R(P)$$
and so
$$ |d_K(P)|\ge R(p)\ge\inf \big\{ \| z-P\|_{\mathcal C}\ :\  z\in\partial K\big\}.$$
This completes the proof of \eqref{suh78100},
as desired.
\end{proof}

\section{The distance function from a graph}

For convenience, we give here two results on the Euclidean
distance function from a graph (the anisotropic case follows
also from this results directly, up to changing constants,
thanks to the equivalency of the norms).
For this, we denote by $d_*$ the distance function $d_K$
in \eqref{dk defin} when $h$ is identically $1$ (hence $\mathcal{C}$
in \eqref{CAL C} is the Euclidean unit ball $B_1$) and $K$ is the portion of
space lying above a function $\zeta\in C^1(\R^{n-1})$, that is $K:=\{ x_n\ge \zeta(x')\}$.

Notice that, in this case, the anisotropic norm $\| \cdot\|_{ \mathcal C }$
in \eqref{pKnorma} is simply the Euclidean norm and, by \eqref{TNAL01},
$d_*$ is simply the signed distance function from the graph of $\zeta$.

Then we have the following results:

\begin{lemma}
Let $b\in\left(0,\frac12\right)$.
Assume that $\zeta(0)=0$ and
that $|\nabla\zeta(x')|\le b$ for every $x'\in\R^{n-1}$ with $|x'|\le 2$.
Then, for any $x\in B_1$ with $x_n\ge \zeta(x')$
$$ d_* (x)\ge 
\frac12\, \big(  x_n-\zeta(x')\big).$$
\end{lemma}

\begin{proof} We let $R:=d_* (x)\ge0$ and we observe that $B_R(x)$
lies above the graph of $\zeta$ and it is tangent to it at some point $z=
(z',z_n)\in \partial B_R(x)$ with $z_n=\zeta(z')$.
We also denote by $\omega$ the interior normal of $B_R(x)$ at $z$. Then, by construction,
\begin{equation}\label{x-zeq}
\frac{x-z}{R}=\omega=\frac{\big(-\nabla\zeta(z'),\,1\big)}{
\sqrt{1+|\nabla\zeta(z')|^2}
} .\end{equation}
Also, since the origin belongs to the graph of $\zeta$,
we have that~$ R=d_*(x)\le |x|\le 1$.
Therefore~$ |z|\le |z-x|+|x|=R+|x|\le 2$.
Accordingly, we deduce from \eqref{x-zeq} that
$$
\frac{|x'-z'|}{R}=\frac{|\nabla\zeta(z')|}{
\sqrt{1+|\nabla\zeta(z')|^2}} \le |\nabla\zeta(z')| \le b.$$
Thus, using again \eqref{x-zeq},
\begin{eqnarray*}&&
1\ge\frac{1}{
\sqrt{1+|\nabla\zeta(z')|^2}
}
=
\frac{x_n-z_n}{R}=
\frac{x_n-\zeta(x')+\zeta(x')-\zeta(z')}{R}\\ &&\qquad\qquad\ge
\frac{x_n-\zeta(x')-b\,|x'-z'|}{R}
\ge \frac{x_n-\zeta(x')-b^2 R}{R}.\end{eqnarray*}
Therefore
\begin{equation*} 
x_n-\zeta(x')\le (1+b^2)R\le 2R= 2d_*(x).\qedhere\end{equation*} 
\end{proof}

Given $r\in\R$, we use the notation $r_-:=\max\{ -r,0\}$.

\begin{lemma} 
Let $\alpha\in(0,1)$, $b\in\left(0,\frac12\right)$ and $r\ge1$,
with~$br^\alpha\le\frac12$.
Assume that $\zeta(0)=0$ and
that $|\nabla\zeta(x')|\le br^\alpha$ for every $x'\in\R^{n-1}$ with $|x'|\le 3r$.

Suppose also that
\begin{equation}\label{x-zeq-33}
{\mbox{$\zeta(x')\ge0$ for every $x'\in\R^{n-1}$.}}\end{equation}
Let $x\in \R^n$ with $|x'|\le r$. Then
$$ \big(d_*(x)\big)_-  \ge \frac12\,\big(x_n-\zeta(x')\big)_-.$$
\end{lemma}

\begin{proof} We can suppose that $x_n < \zeta(x')$, otherwise
$\big(x_n-\zeta(x')\big)_-=0$ and the desired claim is obvious.

Then, we take $R:=- d_*(x)>0$ and we consider the ball $B_R(x)$.
By construction, $B_R(x)$
lies below the graph of $\zeta$ and it is tangent to it at some point $z=
(z',z_n)\in \partial B_R(x)$ with $z_n=\zeta(z')$.
Notice that, in view of \eqref{x-zeq-33},
\begin{equation}\label{x-zeq-55}
z_n \ge0.
\end{equation}
We also denote by $\omega$ the interior normal of $B_R(x)$ at $z$, and so
\begin{equation}\label{x-zeq-2}
\frac{x-z}{R}=\omega=\frac{\big(\nabla\zeta(z'),\,-1\big)}{
\sqrt{1+|\nabla\zeta(z')|^2}
} .\end{equation}
We claim that
\begin{equation}\label{x-zeq-44}
z_n x_n \le br^{\alpha+1}\, z_n.
\end{equation}
Indeed, if $x_n\le0$, then \eqref{x-zeq-44}
follows from \eqref{x-zeq-55}. If instead $x_n>0$, then we know that
$$ \xi(x')=\xi(x')-\xi(0)\le br^\alpha\,|x'|\le br^{\alpha+1}$$
and thus $x_n \in(0,\xi(x'))\subseteq (0,br^{\alpha+1})$, which
gives \eqref{x-zeq-44}.

{F}rom \eqref{x-zeq-55} and \eqref{x-zeq-44} we obtain that
$$ z_n^2 -2x_n z_n \ge z_n^2 -2b r^{\alpha+1} z_n \ge \inf_{t\ge 0} t^2 -2br^{\alpha+1}t 
= -b^2 r^{2(\alpha+1)}.$$
Consequently, using that the origin lies on the graph of $\zeta$,
\begin{eqnarray*}
&& r^2+x_n^2 \ge|x'|^2+x_n^2= |x|^2 \ge |d_*(x)|^2 =|x-z|^2
\\ &&\qquad =|x'-z'|^2 +|x_n-z_n|^2=
|x'-z'|^2 +x_n^2+z_n^2-2x_n z_n\\ &&\qquad
\ge |x'-z'|^2 +x_n^2 -b^2r^{2(\alpha+1)},\end{eqnarray*}
and thus $|x'-z'|^2\le r^2+b^2 r^{2(\alpha+1)}\le 2r^2$.

Therefore $ |z'|\le |x'|+|x'-z'|\le r+\sqrt2\,r\le 3r$. Hence, we deduce from
\eqref{x-zeq-2} that
$$ \frac{|x'-z'|}{R}=\frac{|\nabla\zeta(z')|}{
\sqrt{1+|\nabla\zeta(z')|^2}
} \le |\nabla\zeta(z')|\le br^\alpha$$
and thus
\begin{eqnarray*}
&& 1=\frac{|x-z|}{R}\ge
\frac{z_n-x_n}{R}=\frac{\zeta(z')-\zeta(x')+\zeta(x')-x_n}{R}\\&&\qquad\ge
\frac{-br^\alpha\,|z'-x'|+\zeta(x')-x_n}{R}\ge-b^2r^{2\alpha}+
\frac{\zeta(x')-x_n}{R}.
\end{eqnarray*}
That is,
$$ 
2{ \big(d_*(x)\big)_- }
=2R\ge
(1+b^2r^{2\alpha}) R \ge \zeta(x')-x_n = \big(x_n-\zeta(x')\big)_-,$$
which gives the desired result.
\end{proof}

\section*{Acknowledgements}
This work has been supported by
ERC grant 277749 ``EPSILON Elliptic PDE's and Symmetry of Interfaces and Layers for Odd Nonlinearities".

\end{document}